\documentclass[12pt]{amsart}
\usepackage{amscd,amssymb}
\usepackage{amsthm,amsmath,amssymb}
\usepackage[matrix,arrow,curve]{xy}
\usepackage{mathrsfs}
\usepackage{wasysym}

\newtheorem{theorem}[equation]{Theorem}

\newtheorem{proposition}[equation]{Proposition}
\newtheorem{lemma}[equation]{Lemma}
\newtheorem{corollary}[equation]{Corollary}

\theoremstyle{definition}

\theoremstyle{remark}
\newtheorem{remark}[equation]{Remark}

\makeatletter\@addtoreset{equation}{section} \makeatother

\def\A {\mathfrak{A}}
\def\SS {\mathfrak{S}}
\def\P {\mathbb{P}}

\author{Ivan Cheltsov, Victor Przyjalkowski, Constantin Shramov}

\title{Quartic double solids with icosahedral symmetry}

\address{\emph{Ivan Cheltsov}
\newline
\textnormal{School of Mathematics, The University of Edinburgh,  Edinburgh EH9 3JZ, UK.}
\newline
\textnormal{National Research University Higher School of Economics, Russian Federation, AG Laboratory, HSE, 7 Vavilova str., Moscow, 117312, Russia.}
\newline
\textnormal{\texttt{I.Cheltsov@ed.ac.uk}}}

\address{\emph{Victor Przyjalkowski}
\newline
\textnormal{Steklov Institute of Mathematics, 8 Gubkina street, Moscow 119991, Russia.}
\newline
\textnormal{National Research University Higher School of Economics, Russian Federation, AG Laboratory, HSE, 7 Vavilova str., Moscow, 117312, Russia.}
\newline
\textnormal{\texttt{victorprz@mi.ras.ru}}}

\address{\emph{Constantin Shramov}
\newline
\textnormal{Steklov Institute of Mathematics, 8 Gubkina street, Moscow 119991, Russia.}
\newline
\textnormal{National Research University Higher School of Economics, Russian Federation, AG Laboratory, HSE, 7 Vavilova str., Moscow, 117312, Russia.}
\newline
\textnormal{\texttt{costya.shramov@gmail.com}}}

\begin{document}

\begin{abstract}
We study quartic double solids admitting icosahedral symmetry.
\end{abstract}

\sloppy

\maketitle

\section{Introduction}
\label{section:into}

To study possible embeddings of a finite group $G$ into the Cremona group
$$
\mathrm{Cr}_n(\mathbb{C})=\mathrm{Bir}(\mathbb{P}^n),
$$
one should first describe all $n$-dimensional $G$-Mori fiber spaces (see \cite[Definition~1.1.5]{CheltsovShramov}),
and then to decide which of these $G$-Mori fiber spaces are rational and which are not.
To describe all such embeddings up to conjugation, one should also describe $G$-birational maps between the resulting rational $G$-Mori fiber spaces.
A priori, all three problems (classification, rationality and conjugation) can be solved for any given group~$G$.
In fact, most of the results here are in dimensions
$n=2$ (see \cite{DoIs06} for a nearly complete classification),
and~\mbox{$n=3$} (where only some very special groups have been treated).

In \cite{Pr09}, Prokhorov managed to find all isomorphism classes of
finite \emph{simple} non-abelian subgroups of $\mathrm{Cr}_3(\mathbb{C})$.
He proved that the six groups: $\mathfrak{A}_5$ (alternating group of degree $5$), $\mathrm{PSL}_2(\mathbb{F}_7)$, $\mathfrak{A}_6$,
$\mathrm{SL}_2(\mathbb{F}_8)$, $\mathfrak{A}_7$, and $\mathrm{PSp}_4(\mathbb{F}_3)$
are the only non-abelian finite simple subgroups of $\mathrm{Cr}_3(\mathbb{C})$.
The former three of these six groups actually admit embeddings to $\mathrm{Cr}_2(\mathbb{C})$,
and the icosahedral group $\mathfrak{A}_5$ is also realized as a subgroup of~\mbox{$\mathrm{Cr}_1(\mathbb{C})=\mathrm{PGL}_2(\mathbb{C})$},
while the latter three groups are new three-dimensional artefacts.

The groups $\mathrm{SL}_2(\mathbb{F}_8)$, $\mathfrak{A}_7$, and $\mathrm{PSp}_4(\mathbb{F}_3)$ do not act faithfully
on three-dimensional conic bundles and del Pezzo fibrations, because they do not act faithfully on $\mathbb{P}^1$ and
they do not act faithfully on smooth del Pezzo surfaces (see, for example, \cite{DoIs06}).
Thus, the only $G$-Mori fiber spaces with
$$
G\in\left\{\mathrm{SL}_2(\mathbb{F}_8), \mathfrak{A}_7,\mathrm{PSp}_4(\mathbb{F}_3)\right\}
$$
are $G$-Fano threefolds, i.e., Fano threefolds with terminal singularities such that all \mbox{$G$-invariant} Weil divisors
on them are $\mathbb{Q}$-linearly equivalent to multiples of the anticanonical ones.
In \cite{Pr09}, Prokhorov classified all $G$-Fano threefolds, where $G$ is one of the latter three groups.
His results together with \cite{Beauville} and \cite[Corollary~1.22]{ChSh09b}
imply that $\mathrm{Cr}_3(\mathbb{C})$ contains two (unique, respectively) subgroups isomorphic to $\mathrm{PSp}_4(\mathbb{F}_3)$
(to $\mathrm{SL}_2(\mathbb{F}_8)$ or to $\mathfrak{A}_7$, respectively) up to conjugation.

The papers \cite{ChSh10a} and \cite{ChSh09b} describe several non-conjugate embeddings of the groups
$\mathrm{PSL}_2(\mathbb{F}_7)$ and $\mathfrak{A}_6$ into $\mathrm{Cr}_3(\mathbb{C})$ using a similar approach,
although a complete answer is not known in these two cases.
The book \cite{CheltsovShramov} is devoted to the icosahedral
subgroups in~\mbox{$\mathrm{Cr}_3(\mathbb{C})$}.
In particular, it describes three non-conjugate embeddings
of the group~$\mathfrak{A}_5$ into~\mbox{$\mathrm{Cr}_3(\mathbb{C})$}.

Since quartic double solids are known to have interesting geometrical properties,
it is tempting to study those of them who admit an icosahedral symmetry.
Rational quartic double solids of this kind are of special interest,
since they would provide embeddings of the group $\mathfrak{A}_5$ into $\mathrm{Cr}_3(\mathbb{C})$.
In this paper, we describe all quartic double solids with an action of $\mathfrak{A}_5$,
and study their rationality using the results obtained in the prequel \cite{CheltsovPrzyjalkowskiShramov}.
In particular, we construct one more embedding
$$
\mathfrak{A}_5\hookrightarrow\mathrm{Cr}_3(\mathbb{C}),
$$
and we show that this embedding is not conjugate to any of the three embeddings described in \cite[\S1.4]{CheltsovShramov}.

\medskip

{\bf Acknowledgements.}
This work was started when Ivan Cheltsov and Constantin Shramov were visiting the Mathematisches Forschungsinstitut Oberwolfach
under Research in Pairs program, and was completed when they were visiting Centro Internazionale per la Ricerca Matematica in Trento
under a similar program.
We want to thank these institutions for excellent working conditions,
and personally Marco Andreatta and Augusto Micheletti for hospitality.

This work was supported within the framework of a subsidy granted to the HSE
by the Government of the Russian Federation for the implementation of the Global Competitiveness Program.
Victor Przyjalkowski was supported by the grants RFFI 15-51-50045, RFFI 15-01-02158, RFFI 15-01-02164, RFFI 14-01-00160,
MK-696.2014.1, and NSh-2998.2014.1, and by Dynasty foundation.
Constantin Shramov was supported by the grants  RFFI 15-01-02158,
RFFI 15-01-02164, RFFI 14-01-00160, MK-2858.2014.1, and NSh-2998.2014.1,
and by Dynasty foundation.

\section{Quartic double solids with an action of $\mathfrak{A}_5$}
\label{section:quartic-double-solids}

Let $\tau\colon X\to\mathbb{P}^3$ be double cover branched over a (possibly reducible) reduced quartic surface $S$.

\begin{remark}
\label{remark:Tikhomirov-Voisin}
Recall from \cite[Theorem~5]{Tihomirov3} and \cite[Corollary~4.7(b)]{Voisin88} that $X$ is non-rational provided that the surface $S$ is smooth.
Therefore, we will be mostly interested in the cases when $S$ is singular.
\end{remark}

Suppose that $X$ admits a faithful action of the icosahedral group $\mathfrak{A}_5$.

\begin{remark}
\label{remark:lifting-shifting}
Since $\tau$ is given by the linear system $|-\frac{1}{2}K_X|$,
we see that the action of~$\mathfrak{A}_5$ descends to its action on~$\mathbb{P}^3$ and preserves the quartic $S$.
Vice versa, every \mbox{$\mathfrak{A}_5$-action} on~$\mathbb{P}^3$ that leaves the quartic surface
$S$ invariant can be lifted to the corresponding double cover~$X$.
Indeed, $X$ has a natural embedding to the projectivisation of
the vector bundle
$$\mathcal{E}=\mathcal{O}_{\P^2}\oplus\mathcal{O}_{\P^2}(2).$$
The group $\A_5$ acts on the projective space $\P^3$, so that either
$\A_5$ or its central extension $2.\A_5$ act on the line bundle
$\mathcal{O}_{\P^2}(1)$. In any case, the action of $\A_5$ lifts to the
line bundle $\mathcal{O}_{\P^2}(2)$ and thus to $\mathcal{E}$.
It remains to notice that the group $\A_5$ does not have non-trivial
characters, so that one can write down an $\A_5$-invariant
equation of $X$ in the projectivisation of~$\mathcal{E}$.
\end{remark}

Up to conjugation, the group
$$
\mathrm{Aut}(\mathbb{P}^3)\cong\mathrm{PGL}_4(\mathbb{C})
$$
contains five subgroups isomorphic to $\mathfrak{A}_5$ (cf. \cite[Chapter~VII]{Bli17}).
Namely, denote by $I$ the trivial representation of $\mathfrak{A}_5$.
Let $W_3$ and $W_2^\prime$ be the two irreducible three-dimensional representations of $\mathfrak{A}_5$,
and let $W_4$ be the irreducible four-dimensional representation of~$\mathfrak{A}_5$.
Let $U_2$ and $U_2^\prime$ be two non-isomorphic two-dimensional representations of
the central extension $2.\mathfrak{A}_5$ of the group $\mathfrak{A}_5$ by $\mathbb{Z}/2\mathbb{Z}$,
and let $U_4\cong\mathrm{Sym}^3(U_2)$ be the faithful four-dimensional irreducible representation of $2.\mathfrak{A}_5$.
Then $\mathbb{P}^3$ equipped with a faithful action of the group $\mathfrak{A}_5$
can be identified with one of the following projective spaces: $\mathbb{P}(W_4)$, $\mathbb{P}(I\oplus W_3)$,
$\mathbb{P}(U_2\oplus U_2)$, $\mathbb{P}(U_2\oplus U_2^\prime)$, or $\mathbb{P}(U_4)$.

\begin{remark}
\label{remark:sym-sym}
Computing the symmetric powers  $\mathrm{Sym}^2(U_2\oplus U_2)$ and $\mathrm{Sym}^4(U_2\oplus U_2)$,
we see that the only $\mathfrak{A}_5$-invariant quartic surface in $\mathbb{P}(U_2\oplus U_2)$ is not reduced
(it is the unique $\mathfrak{A}_5$-invariant quadric taken with multiplicity two).
Similarly, we see that there are no $\mathfrak{A}_5$-invariant quartic surfaces in $\mathbb{P}(U_2\oplus U_2^\prime)$ at all.
\end{remark}

Since the quartic surface $S$ is reduced, our $\mathbb{P}^3$ can be identified with either  $\mathbb{P}(U_4)$, or $\mathbb{P}(I\oplus W_3)$, or $\mathbb{P}(W_4)$. Let us start with the case $\mathbb{P}^3=\mathbb{P}(U_4)$.

\begin{theorem}
\label{theorem:U4}
Suppose that $\mathbb{P}^3=\mathbb{P}(U_4)$.
Then $\mathfrak{A}_5$-invariant quartic surfaces in $\mathbb{P}^3$ form a pencil $\mathcal{P}$.
This pencil contains exactly two (isomorphic) surfaces with non-isolated
singularities, and exactly two nodal surfaces.

Let $\mathcal{S}$ be one of the surfaces with non-isolated singularities in
$\mathcal{P}$.
Then the singular locus of $\mathcal{S}$ is a twisted cubic curve $\mathscr{C}$.
Moreover, if~$S=\mathcal{S}$, then there exists a commutative diagram
\begin{equation}
\label{equation:U4}
\xymatrix{
\widetilde{X}\ar@{->}[d]_{\rho}\ar@{->}[rrrr]^{\varphi}\ar@{->}[rrd]^{\widetilde{\tau}}&&&&Q\ar@{->}[d]^{\tau_Q}\\
X\ar@{->}[drr]_{\tau}&&\widetilde{\mathbb{P}}^3\ar@{->}[d]^{\sigma}\ar@{->}[rr]^{\phi}&&\mathbb{P}^2\\
&&\mathbb{P}^3\ar@{-->}[rru]_{\psi}.&&}
\end{equation} %
Here the morphism $\sigma$ is a blow up of the curve $\mathscr{C}$,
the morphism  $\rho$ is a blow up of the preimage of the curve $\mathscr{C}$ on $X$,
the morphism  $\widetilde{\tau}$ is a double cover branched over the proper transform of the surface $\mathcal{S}$ on $\widetilde{\mathbb{P}}^3$,
the rational map $\psi$ is given by the linear system of quadrics in $\mathbb{P}^3$ passing through $\mathscr{C}$,
the surface $Q$ is a smooth quadric, the morphism $\tau_{Q}$ is a double cover branched over the unique $\mathfrak{A}_5$-invariant conic in $\mathbb{P}^2$,
and $\phi$ and $\varphi$ are $\mathbb{P}^1$-bundles.
In particular, in this case $X$ is rational.
\end{theorem}

\begin{proof}
Restrict $U_4$ to the subgroups of $2.\mathfrak{A}_5$ isomorphic to
$2.\mathfrak{A}_4$, $2.\mathrm{D}_{10}$, and $2.\mathrm{D}_{4}$,
where $\mathrm{D}_{2n}$ denotes a dihedral group order $2n$.
We see that none  of these representations has one-dimensional subrepresentations.
Thus, the projective space $\mathbb{P}^3$ does not contain
$\mathfrak{A}_5$-orbits of lengths $5$, $6$, and $15$.
On the other hand, restricting $U_4$ to a subgroup of $2.\A_5$ isomorphic
to $2.\SS_3$,
we see that $\mathbb{P}^3$ contains exactly two $\A_5$-orbits of length $10$.

One has
$$
\mathrm{Sym}^2(U_4)=W_3\oplus W_3^\prime\oplus W_4.
$$
Denote by $\mathcal{Q}$ and $\mathcal{Q}^\prime$ the linear system of quadrics in $\P^3$ that correspond to $W_3$ and $W_3^\prime$, respectively.
Since $\mathbb{P}^3$ does not contain $\mathfrak{A}_5$-orbits of lengths less or equal to eight,
we see that the base loci of $\mathcal{Q}$ and $\mathcal{Q}^\prime$ contain $\mathfrak{A}_5$-invariant curves
$\mathscr{C}$ and $\mathscr{C}^\prime$, respectively.
The degrees of these curves must be less than four.
Since $U_4$ is an irreducible representation of $2.\mathfrak{A}_5$,
we see that both $\mathscr{C}$ and $\mathscr{C}^\prime$ are twisted cubic curves.
This also implies that the base loci of $\mathcal{Q}$ and $\mathcal{Q}^\prime$ are exactly the curves $\mathscr{C}$ and $\mathscr{C}^\prime$, respectively.

The lines in $\mathbb{P}^3$ that are tangent to the curves $\mathscr{C}$ and $\mathscr{C}^\prime$
sweep out quartic surfaces $\mathcal{S}$ and $\mathcal{S}^\prime$, respectively.
These surfaces are $\mathfrak{A}_5$-invariant.
The singular loci of $\mathcal{S}$ and $\mathcal{S}^\prime$ are the curves $\mathscr{C}$ and $\mathscr{C}^\prime$, respectively.
In particular, the surfaces $\mathcal{S}$ and $\mathcal{S}^\prime$ are different.
Their singularities along these curves are locally isomorphic to a product of $\mathbb{A}^1$ and a cusp.

Suppose that $S=\mathcal{S}$. Let us construct the commutative diagram~\eqref{equation:U4}.
Taking a blow up $\sigma\colon\widetilde{\mathbb{P}}^3\to\mathbb{P}^3$ of the curve $\mathscr{C}$,
we obtain an $\mathfrak{A}_5$-equivariant $\P^1$-bundle
$$
\phi\colon\widetilde{\mathbb{P}}^3\to\mathrm{Sym}^2(\mathscr{C})\cong\P^2
$$
that is a projectivisation of a stable rank two vector bundle  $\mathcal{E}$ on $\mathbb{P}^{2}$ with $c_{1}(\mathcal{E})=0$ and
$c_{2}(\mathcal{E})=2$ defined by the exact sequence
$$
0\longrightarrow\mathcal{O}_{\mathbb{P}_{2}}(-1)^{{}\oplus 2}\longrightarrow
\mathcal{O}_{\mathbb{P}_{2}}^{{}\oplus 4}\longrightarrow \mathcal{E}\otimes\mathcal{O}_{\mathbb{P}_{2}}(1)\longrightarrow 0,%
$$
see \cite[Application~1]{SzWi90} for details.

The fibers of $\phi$ are proper transforms of the secant or tangent lines to $\mathscr{C}$.
Moreover, the proper transforms of the tangent lines to $\mathscr{C}$ are mapped by $\phi$ to the points
of the unique $\mathfrak{A}_5$-invariant conic $C$ in $\mathbb{P}^2$.
Let $\tau_Q\colon Q\to\P^2$ be a double cover branched over $C$.
Then $Q$ is a smooth quadric surface.
A preimage on $X$ of a secant line to $\mathscr{C}$ splits as a union of two smooth rational curves,
while a preimage of a tangent line to $\mathscr{C}$ is contained in the ramification locus of $\tau$.
This shows the existence of the commutative diagram~\eqref{equation:U4} and the rationality of $X$.
The same construction applies to the case $S=\mathcal{S}^\prime$.

Denote by $\mathcal{P}$ the pencil generated by $\mathcal{S}$ and $\mathcal{S}^\prime$.
Computing $\mathrm{Sym}^4(U_4)$, we see that $\mathcal{P}$ contains all $\mathfrak{A}_5$-invariant quartic surfaces in $\mathbb{P}^3$.
Since $\mathscr{C}$ is projectively normal, there is an exact sequence of $2.\mathfrak{A}_5$-representations
$$
0\to H^0(\mathcal{O}_{\mathbb{P}^3}(4)\otimes\mathcal{I}_{\mathscr{C}})\to
H^0(\mathcal{O}_{\mathbb{P}^3}(4))\to H^0(\mathcal{O}_{\mathscr{C}}\otimes\mathcal{O}_{\mathbb{P}^3}(4))\to 0,%
$$
where $\mathcal{I}_{\mathscr{C}}$ is the ideal sheaf of $\mathscr{C}$.
The $2.\mathfrak{A}_5$-representation $H^0(\mathcal{O}_{\mathscr{C}}\otimes\mathcal{O}_{\mathbb{P}^3}(4))$
contains a one-dimensional subrepresentation corresponding to a unique $\mathfrak{A}_5$-orbit of length $12$ in~$\mathscr{C}\cong\mathbb{P}^1$.
This shows that $\mathcal{P}$ contains a surface that does not pass through $\mathscr{C}$,
so that $\mathscr{C}$ is not contained in the base locus of $\mathcal{P}$.
Similarly, we see that $\mathscr{C}^\prime$ is not contained in the base locus of $\mathcal{P}$.

Now we suppose that $S\ne\mathcal{S}$ and $S\ne\mathcal{S}^\prime$.
Let 
$S$ be singular.
Note that $S$ is irreducible, because $\P^3$ contains no $\mathfrak{A}_5$-invariant surfaces of degree less than four.

We claim that $S$ has isolated singularities.
Indeed, suppose that $S$ is singular along some $\mathfrak{A}_5$-invariant curve $Z$.
A general plane section of $S$ is an irreducible plane quartic curve that has $\mathrm{deg}(Z)$ singular points,
which implies that the degree of $Z$ is at most three.
Thus, $Z$ is a twisted cubic curve, so that either $Z=\mathscr{C}$ or $Z=\mathscr{C}^\prime$.
Since neither of these curves is contained in the base locus of $\mathcal{P}$,
this would imply that either $S=\mathcal{S}$ or $S=\mathcal{S}^\prime$.
The latter is not the case by assumption.

We see that the singularities of $S$ are isolated.
Hence, $S$ contains at most two non-Du Val singular points (cf. \cite{Degtyarev}).
This follows from \cite[Theorem~1]{Umezu} or from Shokurov's famous \cite[Theorem~6.9]{Shokurov3}
applied to the minimal resolution of singularities of the surface~$S$.
Since the set of all non-Du Val singular points of the surface~$S$ must be $\mathfrak{A}_5$-invariant,
we see that $S$ has none of them, because $U_4$ is an irreducible representation of the group~$2.\mathfrak{A}_5$.
Thus, all singularities of $S$ are Du Val.

By \cite[Lemma~6.7.3(iii)]{CheltsovShramov}, the surface $S$ is nodal,
the set $\mathrm{Sing}(S)$ consists of one $\mathfrak{A}_5$-orbit, and
$$
\big|\mathrm{Sing}(S)\big|\in\big\{5,6,10,12,15\big\}.
$$
Since $\mathbb{P}^3$ does not contain $\mathfrak{A}_5$-orbits
of lengths $5$, $6$, and $15$,
we see that $\mathrm{Sing}(S)$ is either
an $\A_5$-orbit of length $10$ or an $\A_5$-orbit of length $12$.

Suppose that the singular locus of $S$ is an $\A_5$-orbit
$\Sigma_{12}$ of length $12$.
Then $S$ does not contain other $\mathfrak{A}_5$-orbits of length $12$ by \cite[Lemma~6.7.3(iv)]{CheltsovShramov}.
Since $\mathscr{C}$ is not contained in the base locus of $\mathcal{P}$,
and $\mathscr{C}$ is contained in $\mathcal{S}$,
we see that~\mbox{$\mathscr{C}\not\subset S$}.
Since
$$S\cdot\mathscr{C}=12$$
and $\Sigma_{12}$ is the only $\mathfrak{A}_5$-orbit of length at most $12$ in~\mbox{$\mathscr{C}\cong\mathbb{P}^1$},
we have~\mbox{$S\cap\mathscr{C}=\Sigma_{12}$}.
Thus,
$$
12=S\cdot\mathscr{C}\geqslant \sum_{P\in\Sigma_{12}}\mathrm{mult}_{P}(S)=2|\Sigma_{12}|=24,
$$
which is absurd.

Therefore, we see that the singular locus of $S$
is an $\A_5$-orbit $\Sigma_{12}$ of length $10$.
Vice versa, if an $\A_5$-invariant quartic surface
passes through an $\A_5$-orbit of length $10$, then
it is singular by~\cite[Lemma~6.7.1(ii)]{CheltsovShramov}.
We know that there are exactly
two $\A_5$-orbits of length $10$ in $\P^3$, so that there are
two surfaces $\mathcal{R}$ and $\mathcal{R}^\prime$
that are singular exactly at the points of these two
$\A_5$-orbits, respectively. The above argument
shows that every surface in $\mathcal{P}$ except
$\mathcal{S}$, $\mathcal{S}^\prime$, $\mathcal{R}$ and $\mathcal{R}^\prime$
is smooth.
\end{proof}

It would be interesting to find out
whether the double covers of $\mathbb{P}^3$ branched over
nodal surfaces of the pencil $\mathcal{P}$ are rational or not.

\begin{remark}
\label{remark:pencil}
Let us use notation of the proof of Theorem~\ref{theorem:U4}.
Denote by $B$ the base locus of the pencil $\mathcal{P}$.
We claim that $B$ is an irreducible curve with $24$ cusps, and its normalization has genus nine.
Indeed, let $\rho\colon\widehat{\mathcal{S}}\to\mathcal{S}$ be the normalization of the surface $\mathcal{S}\in\mathcal{P}$ having non-isolated singularities,
let $\widehat{\mathscr{C}}$ be the preimage of the curve $\mathscr{C}$ via $\rho$,
and let $\widehat{B}$ be the preimage of the curve $B$ via $\rho$.
Then the action of $\mathfrak{A}_5$ lifts to $\widehat{\mathcal{S}}$,
one has $\widehat{\mathcal{S}}\cong\mathbb{P}^1\times\mathbb{P}^1$,
and $\rho^*(\mathcal{O}_\mathcal{S}\otimes\mathcal{O}_{\mathbb{P}^3}(1))$ is a divisor of bi-degree~$(1,2)$.
This shows that $\widehat{\mathscr{C}}$ is a divisor of bi-degree $(1,1)$,
and $\widehat{B}$ is a divisor of bi-degree $(4,8)$.
In particular, the action of $\mathfrak{A}_5$ on $\widehat{\mathcal{S}}$ is diagonal by \cite[Lemma~6.4.3(i)]{CheltsovShramov},
so that $\widehat{\mathscr{C}}$ is irreducible by~\mbox{\cite[Lemma~6.4.4(i)]{CheltsovShramov}}.
Note that the curve $\widehat{\mathscr{C}}$ is singular.
Indeed, recall that $\mathcal{S}^\prime$ denotes the second surface in
$\mathcal{P}$ with non-isolated singularities,
and its singular locus is the twisted cubic curve $\mathscr{C}^\prime$.
Since $\mathscr{C}^\prime$ is not contained in $\mathcal{S}$,
we see that the intersection $\mathcal{S}\cap\mathscr{C}^\prime$ is an $\mathfrak{A}_5$-orbit $\Sigma_{12}^\prime$ of length $12$.
Similarly, we see that the intersection $\mathcal{S}^\prime\cap\mathscr{C}$ is an $\mathfrak{A}_5$-orbit $\Sigma_{12}$ of length $12$.
These $\mathfrak{A}_5$-orbits are different,
since two different twisted cubic curves cannot have twelve points in common.
Since $B$ is the scheme theoretic intersection of the surfaces $\mathcal{S}$ and $\mathcal{S}^\prime$,
it must be singular at every point of $\Sigma_{12}\cup\Sigma_{12}^\prime$.
Denote by $\widehat{\Sigma}_{12}$ and $\widehat{\Sigma}_{12}^\prime$  the preimages via $\rho$
of the $\mathfrak{A}_5$-orbits $\Sigma_{12}$ and $\Sigma_{12}^\prime$, respectively.
Then
$$
|\widehat{\Sigma}_{12}|=|\widehat{\Sigma}_{12}^\prime|=12,
$$
and $\widehat{B}$ is singular in every point of $\widehat{\Sigma}_{12}^\prime$.
Moreover, $\widehat{B}$ is smooth away of $\widehat{\Sigma}_{12}^\prime$,
because its arithmetic genus is $21$, and the surface $\widehat{\mathcal{S}}$
does not contain $\mathfrak{A}_5$-orbits of length less than~$12$.
In particular, the genus of the normalization of $\widehat{B}$ (and thus also of $B$) equals nine.
On the other hand, we have
$$
\widehat{B}\cap\widehat{\mathscr{C}}=\widehat{\Sigma}_{12},
$$
because $\widehat{B}\cdot\mathscr{C}=12$ and $\widehat{\Sigma}_{12}\subset\widehat{B}$.
This shows that $B$ is an irreducible curve whose only singularities are the points of $\Sigma_{12}\cup\Sigma_{12}^\prime$,
and each such point is a cusp of the curve $B$.
\end{remark}

Now let us deal with the case when $\mathbb{P}^3$ is identified with $\mathbb{P}(I\oplus W_3)$.

\begin{theorem}
\label{theorem:I+W3}
Suppose that $\mathbb{P}^3=\mathbb{P}(I\oplus W_3)$.
Then there exists a unique $\mathfrak{A}_5$-invariant conic $C$ in $\mathbb{P}^3$.
All $\mathfrak{A}_5$-invariant quadric surfaces in $\mathbb{P}^3$ form a pencil $\mathcal{P}$.
Two general quadrics from this pencil are tangent to each other along $C$.
Moreover, any reduced $\mathfrak{A}_5$-invariant quartic surface in $\mathbb{P}^3$ is a union of two different quadrics from $\mathcal{P}$.
Furthermore, if $S$ is such a quartic surface, then there exists a commutative diagram
\begin{equation}
\label{equation:I+W3}
\xymatrix{
X\ar@{->}[d]_{\tau}&&\widetilde{X}\ar@{->}[d]_{\widetilde{\tau}}\ar@{->}[ll]_{\rho}&&\widehat{X}\ar@{->}[d]_{\widehat{\tau}}\ar@{->}[ll]_{\widetilde{\rho}}\ar@{->}[drr]^{\varphi}&&\\
\mathbb{P}^3\ar@{-->}[rrrrrrd]_{\psi}&&\widetilde{\mathbb{P}}^3\ar@{->}[ll]_{\sigma}&&\widehat{\mathbb{P}}^3\ar@{->}[ll]_{\widetilde{\sigma}}\ar@{->}[drr]^{\phi}&&\mathbb{P}^1\ar@{->}[d]^{\tau_{\mathbb{P}^1}}\\
&&&&&&\mathbb{P}^1.}
\end{equation} %
Here the morphism $\sigma$ is a blow up of the conic $C$,
the morphism $\rho$ is a blow up of the preimage of $C$ on $X$ (which is the singular locus of $X$),
the morphism $\widetilde{\tau}$ is a double cover branched over the proper transform of the surface $S$ on $\widetilde{\mathbb{P}}^3$,
the morphism $\widetilde{\sigma}$ is a blow up of the intersection curve $\widetilde{C}$ of the proper transforms
on $\widetilde{\mathbb{P}}^3$ of the irreducible components of $S$,
the morphism $\widetilde{\rho}$ is a blow up of the preimage of $\widetilde{C}$ on $\widetilde{X}$ (which is the singular locus of~$\widetilde{X}$),
the morphism $\widehat{\tau}$ is a double cover branched over the proper transform of the surface~$S$ on~$\widehat{\mathbb{P}}^3$,
the rational map $\psi$ is given by the pencil $\mathcal{P}$,
the morphism $\tau_{\mathbb{P}^1}$ is a double cover, and general fibers of $\phi$ and $\varphi$ are smooth quadric surfaces.
In particular, the threefold $X$ is rational.
\end{theorem}

\begin{proof}
Straightforward and left to the reader.
\end{proof}

We will deal with  the remaining case $\mathbb{P}^3=\mathbb{P}(W_4)$ in the next section.

\section{Hashimoto's pencil of quartic surfaces}
\label{section:Hashimoto}

Recall that $W_4$ denotes the irreducible four-dimensional representation of $\mathfrak{A}_5$.
Put~\mbox{$\mathbb{P}^3=\mathbb{P}(W_4)$}.
Note that the sum of $W_4$ and the trivial $\mathfrak{A}_5$-representation $I$
is the usual five-dimensional permutation representation of $\mathfrak{A}_5$.
Let $x_0,\ldots,x_4$ be the coordinates in $I\oplus W_4$ permuted by $\mathfrak{A}_5$.
Then $W_4$ is given in $I\oplus W_4$ by the equation
$$
x_0+x_1+x_2+x_3+x_4=0.
$$
Thus, we can consider $x_0,\ldots,x_4$ as homogenous coordinates on $\mathbb{P}^3\cong\mathbb{P}(W_4)$,
keeping in mind that $x_0=-(x_1+\ldots+x_4)$.

For every positive $i$, put
$$
F_i=x_0^i+x_1^i+x_2^i+x_3^i+x_4^i=\Big(-(x_1+\ldots+x_4)\Big)^i+x_1^i+x_2^i+x_3^i+x_4^i.
$$
Then $F_2$ is the unique $\mathfrak{A}_5$-invariant polynomial of degree two in $x_1,x_2,x_3,x_4$ modulo scaling.
Similarly, modulo scaling, $F_3$ is the unique $\mathfrak{A}_5$-invariant polynomial of degree three in $x_1,x_2,x_3,x_4$.
Finally, modulo $F_2$ and scaling, $F_4$ is the unique $\mathfrak{A}_5$-invariant polynomial of degree four in $x_1,x_2,x_3,x_4$.
Thus, any $\mathfrak{A}_5$-invariant quartic surface in $\mathbb{P}^3$ is given by
\begin{equation}
\label{equation:quartic}
F_4=\lambda F_2^2
\end{equation}
for some $\lambda\in\mathbb{C}$.

\begin{remark}
\label{remak:Hashimoto}
The quartic surfaces given by \eqref{equation:quartic} have been studied by Kenji Hashimoto in \cite{Hashimoto11}.
In particular, he described all singular surfaces in this pencil.
\end{remark}

The lengths of $\mathfrak{A}_5$-orbits in $\mathbb{P}^3$ are $5$, $10$, $12$, $15$, $20$, $30$, and $60$
(see, for example,~\mbox{\cite[Corollary~5.2.3]{CheltsovShramov}}).
Let $\Sigma_{5}$, $\Sigma_{10}$, $\Sigma_{10}^\prime$, $\Sigma_{15}$ be the $\mathfrak{A}_5$-orbits
of the points
$$
[-4:1:1:1:1],\ [0:0:0:-1:1],\ [-2:-2:-2:3:3],\ [0:-1:-1:1:1],
$$
respectively. Then $|\Sigma_{5}|=5$, $|\Sigma_{10}|=|\Sigma_{10}^\prime|=10$, and $|\Sigma_{15}|=15$.

\begin{remark}
\label{remark:10-points}
The points $[0:0:0:-1:1]$, $[0:0:-1:0:1]$, and $[0:0:-1:1:0]$ are collinear.
Thus, there are ten lines in $\mathbb{P}^3$ such that each of them contains three points of $\Sigma_{10}$,
and each point of $\Sigma_{10}$ lies on exactly three of these lines.
Similarly, the points $[-2:-2:-2:3:3]$, $[-2:-2:3:-2:3]$, $[-2:-2:3:3:-2]$, and $[3:3:-2:-2:-2]$ are coplanar.
Hence, there are ten planes in $\mathbb{P}^3$ such that each of them contains four points of $\Sigma_{10}^\prime$,
and each point of $\Sigma_{10}^\prime$ lies on exactly four of these planes.
In particular, for each $\mathfrak{A}_5$-orbit $\Sigma\in\{\Sigma_{10}, \Sigma_{10}^\prime\}$,
there exists a plane in $\mathbb{P}^3$ that contains at least four points of $\Sigma$
such that no three of them are collinear.
\end{remark}

The $\mathfrak{A}_5$-orbits $\Sigma_{5}$, $\Sigma_{10}$, $\Sigma_{10}^\prime$, and $\Sigma_{15}$
are the only $\mathfrak{A}_5$-orbits of lengths $5$, $10$, and $15$ in~$\mathbb{P}^3$.
Moreover, there are exactly two $\mathfrak{A}_5$-orbits $\Sigma_{12}$ and $\Sigma_{12}^\prime$ in $\mathbb{P}^3$ of length $12$,
see, for example, \cite[Corollary~5.2.3]{CheltsovShramov} and \cite[Lemma~5.3.3]{CheltsovShramov}.

By \cite[Lemma~5.3.3(xi)]{CheltsovShramov}, the curve in $\mathbb{P}^3$ that is given by $F_2=F_4=0$  is a smooth curve of genus nine.
In particular, this curve does not contain the $\mathfrak{A}_5$-orbits $\Sigma_{5}$, $\Sigma_{10}$, $\Sigma_{10}^\prime$,
and $\Sigma_{15}$, see, for example, \cite[Lemma~5.1.4]{CheltsovShramov}.
Therefore, there exists a unique $\mathfrak{A}_5$-invariant quartic surface $S_5$ (respectively, $S_{10}$, $S_{10}^\prime$, $S_{15}$) in $\mathbb{P}^3$
that contains the $\mathfrak{A}_5$-orbit~$\Sigma_{5}$ (respectively, $\Sigma_{10}$, $\Sigma_{10}^\prime$, $\Sigma_{15}$).

\begin{proposition}[{\cite[Proposition~3.1]{Hashimoto11}}]
\label{proposition:Hashimoto}
The surface $S_5$ (respectively, $S_{10}$, $S_{10}^\prime$, $S_{15}$) is given by the equation \eqref{equation:quartic}
with $\lambda=\frac{13}{20}$ (respectively, $\lambda=\frac{1}{2}$, $\lambda=\frac{7}{30}$, $\lambda=\frac{1}{4}$).
The surface~$S_5$ (respectively, $S_{10}$, $S_{10}^\prime$, $S_{15}$) has nodes
at the points of the $\mathfrak{A}_5$-orbit $\Sigma_{5}$ (respectively, $\Sigma_{10}$, $\Sigma_{10}^\prime$, $\Sigma_{15}$), and is smooth outside of it.
Moreover, $S_5$, $S_{10}$, $S_{10}^\prime$, and $S_{15}$ are the only singular surfaces given by  \eqref{equation:quartic}.
\end{proposition}

Denote by $X_5$ (respectively, $X_{10}$, $X_{10}^\prime$, $X_{15}$) the double cover of $\mathbb{P}^3$ branched over
the surface $S_5$ (respectively, $S_{10}$, $S_{10}^\prime$, $S_{15}$).
As we already mentioned in Remark~\ref{remark:lifting-shifting}, the action of the group $\mathfrak{A}_5$ lifts to each of these threefolds.
In the next two theorems, we prove two guesses that were formulated in \cite[Example~1.3.9]{CheltsovShramov}

\begin{theorem}
\label{theorem:X5}
The threefold $X_5$ is $\mathbb{Q}$-factorial and non-rational.
Moreover, twisted cubics passing through $\Sigma_5$ define a commutative diagram
$$
\xymatrix{
&V^\prime\ar@{-->}[rr]^{\rho^\prime}\ar@{->}[ld]_{\pi^\prime}\ar@{->}[rd]_{\zeta^\prime}&&W^\prime\ar@{->}[ld]^{\xi^\prime}\ar@{->}[rdd]^{\phi^\prime}&\\%
X_5\ar@{->}[d]_{\tau}&&V_{3}^\prime\ar@{->}[d]_{\iota}&&\\
\mathbb{P}^3&&V_3&&Y_5\\
&V\ar@{-->}[rr]_{\rho}\ar@{->}[lu]^{\pi}\ar@{->}[ru]^{\zeta}&&W\ar@{->}[lu]_{\xi}\ar@{->}[ru]_{\phi}.&}
$$ %
Here the morphism $\tau$ is a double cover branched over the surface $S_5$,
the morphism $\pi$ is the blow-up of $\Sigma_5$, the morphism $\pi^\prime$ is the blow-up of the singular locus of $X_5$,
the morphism $\zeta$ is the contraction of the proper transforms of the ten lines
$L_1,\ldots,L_{10}$ in $\mathbb{P}^3$ passing through pairs of points in $\Sigma_5$,
the rational map $\rho$ is a composition of Atiyah flops in these ten curves,
the morphism $\zeta^\prime$ is the contraction of the proper transforms of the twenty curves
that are mapped by $\tau$ to the lines $L_1,\ldots,L_{10}$,
the rational map $\rho^\prime$ is a composition of Atiyah flops in these twenty curves,
the variety $V_3$ is the Segre cubic hypersurface in $\mathbb{P}^4$,
the variety $V_3^\prime$ is a complete intersection of the cone over $V_3$ and a quadric hypersurface in~$\mathbb{P}^5$,
the morphism $\iota$ is the natural double cover given by the projection of this cone from its vertex,
both morphisms $\phi$ and $\phi^\prime$ are conic bundles, and $Y_5$ is a del Pezzo surface of degree $5$.
In particular, the $\mathfrak{A}_5$-Mori fiber space $\phi^\prime\colon W^\prime\to Y_5$ is $\mathfrak{A}_5$-birational to~$X_{5}$.
\end{theorem}

\begin{proof}
The threefold $X_5$ is $\mathbb{Q}$-factorial by \cite[Theorem 1.8]{CheltsovPrzyjalkowskiShramov},
and it is non-rational by \mbox{\cite[Theorem~1.2]{CheltsovPrzyjalkowskiShramov}}.
The preimage on $X_5$ of a general twisted cubic that passes through $\Sigma_5$ is an irreducible (singular) rational curve.
Thus, the existence of the commutative diagram follows from \cite[Proposition~4.7]{Pr10}.
\end{proof}

In the proof of \cite[Theorem~4.2]{CheltsovPrzyjalkowskiShramov},
the authors constructed a birational transformation of $X_5$ into a standard conic bundle over the smooth del Pezzo surface $Y_5$ of degree $5$.
Unlike the birational map $\rho^\prime\circ(\pi^\prime)^{-1}$ from Theorem~\ref{theorem:X5},
this transformation is not $\mathfrak{A}_5$-equivariant (it is only $\mathfrak{A}_4$-equivariant).

\begin{theorem}
\label{theorem:X10}
The threefolds $X_{10}$ and $X_{10}^\prime$ are $\mathbb{Q}$-factorial, and $X_{10}$ is not stably rational.
\end{theorem}

\begin{proof}
The $\mathbb{Q}$-factoriality of $X_{10}$ (respectively,  $X_{10}^\prime$) is equivalent to the fact
that  $\Sigma_{10}$  (respectively,  $\Sigma_{10}^\prime$) is not contained in a quadric surface in $\mathbb{P}^3$,
see \cite[\S3]{Clemens} or \cite{En99}.
The latter condition is very easy to check by solving the system of linear equations.
However, it can also be easily proved without any computations.
Indeed, suppose that there are quadrics in $\mathbb{P}^3$
passing through $\Sigma_{10}$.
Note that $\Sigma_{10}$  does not lie on the quadric
given by $F_2=0$ by \cite[Lemma~5.3.3(vi)]{CheltsovShramov}.
The latter equation defines the unique $\mathfrak{A}_5$-invariant
quadric in $\mathbb{P}^3$.
Thus, the intersection of all
quadrics passing through $\Sigma_{10}$ is either an
$\mathfrak{A}_5$-invariant set of at most eight points, or an
$\mathfrak{A}_5$-invariant curve $Z$ of degree at most four.
The former case is clearly impossible because the set $\Sigma_{10}$
contains more than eight points. In the latter case one has
$\deg(Z)=4$, because there are no $\mathfrak{A}_5$-invariant
curves of degree at most three in $\mathbb{P}^3$ by
\cite[Lemma~5.3.3(ix)]{CheltsovShramov}. Thus, $Z$ is a complete intersection
of two quadrics. On the other hand, a direct computation shows that
$\mathrm{Sym}^2(W_4)$ does not contain two-dimensional
subrepresentations of $\mathfrak{A}_5$.
The obtained contradiction shows that
$\Sigma_{10}$ is not contained in any quadric surface
in $\mathbb{P}^3$.
Similarly, we see that $\Sigma_{10}^\prime$ is not contained in any quadric surface in $\mathbb{P}^3$.
Hence, both threefolds $X_{10}$ and $X_{10}^\prime$ are $\mathbb{Q}$-factorial.

Let us show that $X_{10}$ is not stably rational.
To do this, pick a point $O\in\Sigma_{10}$,
denote by $\tau$ the double cover $X_{10}\to\mathbb{P}^3$ that is branched at $S_{10}$.
Then there exists a commutative diagram
$$
\xymatrix{
X_{10}\ar@{->}[rr]^{\tau}&&\mathbb{P}^3\ar@{-->}[d]^{p_{O}}\\
\widetilde{X}_{10}\ar@{->}[u]^{f}\ar@{->}[rr]_{\pi}&&\mathbb{P}^{2},}
$$
where the morphism $p_{O}$ is the linear projection from the point $O$,
the morphism $f$ is the blow up of the point in $X_{10}$ that is mapped to the point $O$ by the double cover $\tau$,
and $\pi$ is a conic bundle.

Let us describe the degeneration curve $C$ of the conic bundle $\pi$.
We may assume that
$$
O=[0:0:0:-1: 1].
$$
Plugging $x_0=-x_1-x_2-x_3-x_4$ into the equation \eqref{equation:quartic} with $\lambda=\frac{1}{2}$,
and considering the affine equation of $S_{10}$ in the chart $x_4\ne 0$ with the new coordinates
$$
y_1=\frac{x_1}{x_4},\quad y_2=\frac{x_2}{x_4},\quad  y_3=\frac{x_3}{x_4}-1,
$$
we see that $O=(0,0,0)$ in these coordinates, and the chart of the surface $S_{10}$ in $\mathbb{A}^3$ is given by the equation
\begin{equation}
\label{equation:affine}
(y_1+y_2+y_3+2)^4+y_1^4+y_2^4+(y_3+1)^4+1=\frac{1}{2}\left((y_1+y_2+y_3+2)^2+y_1^2+y_2^2+(y_3+1)^2+1\right)^2.
\end{equation}
Then every line $L$ in $\mathbb{A}^3$ passing through $O$ is given by
$$
(y_1,y_2,y_3)=t(z_1,z_2,z_3)
$$
for some $(z_1,z_2,z_3)\ne (0,0,0)$, where $t$ is a parameter.
Plugging this parametric equation into \eqref{equation:affine},
dividing the resulting equation by $t^2$, and taking the discriminant of the resulting quadratic equation,
we see that the equation of $C$ in $\mathbb{P}^2$ is
$$
z_3\left(z_3(z_1+z_2)+z_1^2+z_1z_2+z_2^2\right)(z_2z_3^2+z_1z_3^2+z_1^2z_3+5z_1z_2z_3+4z_1^2z_2+z_2^2z_3+4z_1z_2^2)=0,
$$
where we consider $z_1$, $z_2$, and $z_3$ as homogeneous coordinates on $\mathbb{P}^2$.
Thus $C$ is a union of a line $\ell$ that is given by $z_3=0$, a smooth conic $\gamma$ that is given by
$$
z_3(z_1+z_2)+z_1^2+z_1z_2+z_2^2=0,
$$
and a smooth cubic curve $\zeta$ that is given by
$$
z_2z_3^2+z_1z_3^2+z_1^2z_3+5z_1z_2z_3+4z_1^2z_2+z_2^2z_3+4z_1z_2^2=0.
$$
The line $\ell$ intersects the curves $\gamma$ and $\zeta$ transversally.
Moreover, the curves $\gamma$ and $\zeta$ are tangent at three points $[0:1:-1]$, $[1:0:-1]$, and $[0:0:1]$,
so that $\gamma\cup\zeta$ has three tacnodes at these points (cf. Remark~\ref{remark:10-points} and \cite[Proposition~3.2(i)]{CheltsovPrzyjalkowskiShramov}).
Furthermore, no point in $\Sigma_{10}$ is mapped to a point in $\ell\cap\gamma$ by $p_{O}$,
because all points in $\Sigma_{10}$ are defined over~$\mathbb{Q}$,
and the two points of the intersections $\ell\cap\gamma$ are $[-1+\sqrt{3}:2:0]$ and $[-1+\sqrt{3}:2:0]$.

By \cite[Theorem~4.2(i),(ii)]{CheltsovPrzyjalkowskiShramov}, there exists a commutative diagram
$$
\xymatrix{
V\ar@{->}[d]_{\nu}\ar@{-->}[rr]^{\rho}&&\widetilde{X}_{10}\ar@{->}[d]^{\pi}\\
U\ar@{->}[rr]^{\varrho}&&\mathbb{P}^2.}
$$
Here $V$ is a smooth projective threefold, $U$ is a smooth surface,
the relative Picard group of $V$ over $U$ has rank $1$, and
$\varrho$ is a birational morphism that factors as
$$
U\stackrel{\varrho_t'}\longrightarrow U_t\stackrel{\varrho_t}\longrightarrow U_n\stackrel{\varrho_n}\longrightarrow\P^2,
$$
where the morphism $\varrho_n$ is a blow up of the three points of $\zeta\cap\ell$,
the morphism $\varrho_t$ is a blow up of the three points of $\gamma\cap\zeta$,
and $\varrho_t^\prime$ is a blow up of the three intersection points of the
proper transforms of the curves $\gamma$ and $\zeta$ on the surface $U_t$.

Let $\Delta$ be the degeneration curve of the conic bundle $\nu$.
Then $\Delta$ is the proper transform of the curve $C$ by \cite[Theorem~4.2(iii)]{CheltsovPrzyjalkowskiShramov}.
Thus, the curve $\Delta$ is not connected, so that $H^3(V,\mathbb{Z})$ has non-trivial $2$-torsion by \cite[Theorem~2]{Zagorski},
see also \cite[Proposition~3]{ArtinMumford}.
So the threefold $X_{10}$ is not stably rational by \cite[Proposition~1]{ArtinMumford}.
\end{proof}

We do not know whether $X_{10}^\prime$ is rational or not.

\begin{remark}
\label{remark:X-10-prime}
The proof of non-rationality of $X_{10}$ is not applicable to $X_{10}^\prime$.
Indeed, arguing as in the proof of Theorem~\ref{theorem:X10}, we obtain a commutative diagram
$$
\xymatrix{
X_{10}^\prime\ar@{->}[rr]^{\tau}&&\mathbb{P}^3\ar@{-->}[d]^{p_{O}}\\
\widetilde{X}_{10}^\prime\ar@{->}[u]^{f}\ar@{->}[rr]_{\pi}&&\mathbb{P}^{2},}
$$
where the morphism $\tau$ is the double cover branched at $S_{10}^\prime$,
the rational map $p_{O}$ is the linear projection from the point
$$
O=[-2:-2:-2:3:3]\in \Sigma_{10}^\prime,
$$
the morphism $f$ is the blow up of the point in $X_{10}^\prime$ that is mapped to $O$,
and $\pi$ is a conic bundle. Denote by $C$ the degeneration curve of the conic bundle $\pi$.
Making computations similar to those in the proof of Theorem~\ref{theorem:X10}, we see that $C$ can be given by
\begin{multline*}
-16z_1^6-16z_2^6-13z_2^2z_3^4-13z_1^2z_3^4-42z_2^3z_3^3-61z_2^4z_3^2-\\
-42z_1^3z_3^3-61z_1^4z_3^2+12z_1^4z_2^2+104z_1^3z_2^3+12z_1^2z_2^4-48z_1^5z_2-48z_1^5z_3-48z_1z_2^5-\\
-48z_2^5z_3+93z_1^2z_2^2z_3^2-26z_1^3z_2z_3^2-12z_1^2z_2z_3^3-12z_1z_2^2z_3^3-72z_1^4z_2z_3+120z_1^2z_2^3z_3-\\
-72z_1z_2^4z_3-26z_1z_2^3z_3^2-10z_1z_2z_3^4+120z_1^3z_2^2z_3=0,
\end{multline*}
where $z_1$, $z_2$, and $z_3$ are homogeneous coordinates on $\mathbb{P}^2$.
One can check that $C$ is an irreducible nodal curve with exactly nine nodes.
It follows from~\mbox{\cite[Theorem~4.2(i),(ii)]{CheltsovPrzyjalkowskiShramov}} that there exists a commutative diagram
$$
\xymatrix{
V\ar@{->}[d]_{\nu}\ar@{-->}[rr]^{\rho}&&\widetilde{X}_{10}^\prime\ar@{->}[d]^{\pi}\\
U\ar@{->}[rr]^{\varrho}&&\mathbb{P}^2,}
$$
where $V$ is a smooth projective threefold, $U$ is a smooth surface,
the relative Picard group of $V$ over $U$ has rank $1$, and
$\varrho$ is a blow up of nine nodes of $C$.
Moreover, the degeneration curve of the conic bundle $\nu$ is the proper transform of the curve $C$
by \cite[Theorem~4.2(iii)]{CheltsovPrzyjalkowskiShramov}.
Thus, $\Delta$ is an irreducible smooth elliptic curve in $|-2K_U|$,
which implies that the group $H^3(V,\mathbb{Z})$ is trivial (see, for example, \cite[Theorem~2]{Zagorski}).
In particular, the intermediate Jacobian of $V$ is trivial,
and the approach of \cite{ArtinMumford} does not work in this case either.
Note that \cite[Conjecture~10.3]{Shokurov} predicts that $X_{10}^\prime$ is
non-rational.
\end{remark}

\begin{remark}
\label{remark:Endrass}
In \cite[p.~354]{En99}, it is claimed that the resolution of singularities of any \mbox{$\mathbb{Q}$-factorial}
nodal quartic double solid with ten nodes has non-trivial torsion in the third integral cohomology group.
The computations in Remark~\ref{remark:X-10-prime} show that $X_{10}^\prime$ is a counter-example to this claim.
\end{remark}

\begin{remark}
\label{remark:Zagorskii}
In the proof of Theorem~\ref{theorem:X10} and in Remark~\ref{remark:X-10-prime}, we refer to \cite[Theorem~2]{Zagorski}.
Note that the notation of this theorem is a bit non-standard.
Namely, the second summand on the right hand side in \cite[(13)]{Zagorski} is torsion free (see \cite[Lemma~5]{Zagorski} and~\mbox{\cite[Lemma~7]{Zagorski}} for a detailed computation).
\end{remark}

Recall that a normal variety $X$ with an action of the finite group $G$ is said to be $G\mathbb{Q}$-factorial if any $G$-invariant Weil divisor on
$X$ is a $\mathbb{Q}$-Cartier divisor.

\begin{theorem}
\label{theorem:X15}
The threefold $X_{15}$ is not $\mathbb{Q}$-factorial, it is $\mathfrak{A}_5\mathbb{Q}$-factorial, and it is rational.
\end{theorem}

\begin{proof}
The threefold $X_{15}$ is not $\mathbb{Q}$-factorial by \cite[Corollary~1.7]{CheltsovPrzyjalkowskiShramov}.
Its rationality follows from  \cite[Theorem~8.1]{Prokhorov} or\cite[Theorem~1.3]{CheltsovPrzyjalkowskiShramov}.

Let us show that $X_{15}$ is $\mathfrak{A}_5\mathbb{Q}$-factorial.
Put
$$
O=[0:-1:-1:1:1],
$$
and denote by $\Gamma$ its stabilizer in $\mathfrak{A}_5$. Then $O\in\Sigma_{15}$, and $\Gamma\cong(\mathbb{Z}/2\mathbb{Z})^2$.
To prove that $X_{15}$ is $\mathfrak{A}_5\mathbb{Q}$-factorial, it is enough to show that the $\Gamma$-invariant local class group of the point $O$ is trivial.

The quartic $S_{15}$ is given by the equation $R=0$, where
$$
R=x_0^4+x_1^4+x_2^2+x_3^2+x_4^4-\frac{1}{4}\big(x_0^2+x_1^2+x_2^2+x_3^2+x_4^2\big)^2,
$$
and $x_0=-(x_1+x_2+x_3+x_4)$. Let $\sigma$ be the element of $\mathfrak{A}_5$ acting by
$$
\sigma(x_0)=x_0,\ \sigma(x_1)=x_3,\ \sigma(x_2)=x_4,\ \sigma(x_3)=x_1,\ \sigma(x_4)=x_2.
$$
Then $\sigma\in \Gamma$. Consider new homogeneous coordinates $y_1,\ldots,y_4$ in $\P^3$ such that
$$
x_1=y_1-y_3-y_4,\ x_2=y_2-y_4,\ x_3=y_1+y_3+y_4,\ x_4=y_2+y_4.
$$
Then $x_0=-2(y_1+y_2)$, and the point $O$ is $[0:0:0:1]$ in the new coordinates. One has
$$
\sigma(y_1)=y_1,\ \sigma(y_2)=y_2,\ \sigma(y_3)=-y_3,\ \sigma(y_4)=-y_4.
$$
Write
$$
R=R_2(y_1,y_2,y_3)y_4^2+R_3(y_1,y_2,y_3)y_4+R_4(y_1,y_2,y_3),
$$
where $R_i$ is a form of degree $i$. Then
$$
R_2(y_1,y_2,y_3)=4y_3^2-16y_1y_2.
$$
The threefold $X_{15}$ is given in the weighted projective space $\mathbb{P}(1,1,1,1,2)$
with weighted homogeneous coordinates $y_1, y_2, y_3, y_4$, and $w$ by equation $w^2=R$,
and $\sigma$ acts trivially on $w$.
Identify the point $O\in\mathbb{P}^3$ with the unique point of $X_{15}$ that is mapped to $O$ by the double cover morphism.
Then the tangent cone to $X_{15}$ at $O$ is a cone over a quadric surface $B$ given by equation
$$
16y_1y_2=4y_3^2-w^2
$$
in a three-dimensional projective space with coordinates $y_1, y_2, y_3$, and $w$.
The two lines $y_1=w-2y_3=0$ and $y_1=w+2y_3=0$ are contained in two different pencils of lines on~$B$.
They are interchanged by the involution $\sigma$,
which implies that the $\Gamma$-invariant local class group of the point $O$ is trivial.
\end{proof}

Since $X_{15}$ is rational, the $\mathfrak{A}_5$-action on $X_{15}$ gives an embedding
$$
\mathfrak{A}_5\hookrightarrow\mathrm{Cr}_3(\mathbb{C}).
$$
In the next section we will see that $X_{15}$ is $\mathfrak{A}_5$-birationally superrigid
(see \cite[Definition~3.1.1]{CheltsovShramov}),
so that the latter embedding is not conjugate to the three embeddings described in \cite[\S1.4]{CheltsovShramov}.

\section{$\mathfrak{A}_5$-birational superrigidity}

Let us use notation of Section~\ref{section:Hashimoto}.
Let $S$ be a quartic surface in $\P^3$ that is given by \eqref{equation:quartic},
and let $\tau\colon X\to\P^3$ be a double cover branched over $S$.
By Remark~\ref{remark:lifting-shifting}, the threefold $X$ is faithfully acted on by the group $\mathfrak{A}_5$.
Recall that $Q$ is the surface in $\P^3$ that is given by~$F_2=0$.
The surface $Q$ is smooth, so that $Q\cong\mathbb{P}^1\times\mathbb{P}^1$.
Denote by $H$ the class of the pull-back of the plane in $\mathbb{P}^3$ via $\tau$.

\begin{remark}
\label{remark:G-Fano}
Every $\mathfrak{A}_5$-invariant Weil divisor on $X$ is rationally equivalent to a multiple of $H$.
This follows from Theorems~\ref{theorem:X5}, \ref{theorem:X10}, \ref{theorem:X15}, and the Lefschetz theorem.
\end{remark}

Thus, $X$ is $\mathfrak{A}_5$-Fano threefold. 
The goal of this section is to prove the following result.

\begin{theorem}
\label{theorem:G-rigid}
The threefold $X$ is $\mathfrak{A}_5$-birationally superrigid if and only if $X\ne X_5$.
\end{theorem}

\begin{corollary}\label{corollary:G-rigid-conjugacy}
The group $\mathrm{Cr}_3(\mathbb{C})$ contains at least four non-conjugate subgroups isomorphic to~$\mathfrak{A}_5$.
\end{corollary}

\begin{proof}
Since the threefold  $X_{15}$ is rational by Theorem~\ref{theorem:X15},
the required assertion follows from Theorem~\ref{theorem:G-rigid}, \cite[Remark~1.2.1]{CheltsovShramov}, \cite[Example~1.3.9]{CheltsovShramov},
and \cite[Theorem~1.4.1]{CheltsovShramov}.
\end{proof}

\begin{corollary}\label{corollary:G-rigid-conjugacy-S-5}
The group $\mathrm{Cr}_3(\mathbb{C})$ contains at least three non-conjugate subgroups isomorphic to the permutation group $\mathfrak{S}_5$.
\end{corollary}

\begin{proof}
Note that  $W_4$ is a restriction to $\mathfrak{A}_5$ of the representation of the group $\mathfrak{S}_5$.
Consider the corresponding action of the group $\mathfrak{S}_5$ on $\mathbb{P}^3$.
Since $S_{15}$ is $\mathfrak{S}_5$-invariant, $X_{15}$ is also acted on by $\mathfrak{S}_5$ (cf. Remark~\ref{remark:lifting-shifting}).
Moreover, the quadric $Q\cong\mathbb{P}^1\times\mathbb{P}^1$ is $\mathfrak{S}_5$-invariant.
Thus, we have three rational Fano threefolds acted on by $\mathfrak{S}_5$.
They are $\mathbb{P}^3$, $Q\times\mathbb{P}^1$ (with a trivial action on the second factor), and $X_{15}$.
Since $X_{15}$ is $\mathfrak{A}_5$-birationally superrigid by Theorem~\ref{theorem:G-rigid},
it is also $\mathfrak{S}_5$-birationally superrigid.
Hence, there are no $\mathfrak{S}_5$-birational maps $X_{15}\dasharrow\mathbb{P}^3$ and $X_{15}\dasharrow Q\times\mathbb{P}^1$.
On the other hand, the abelian subgroup
$$
(\mathbb{Z}/2\mathbb{Z})^2\subset\mathfrak{A}_5\subset\mathfrak{S}_5
$$
fixes a point in $\Sigma_{15}\subset\mathbb{P}^3$ and does not have fixed points in $Q$.
Therefore, there is no \mbox{$\mathfrak{S}_5$-birational} map $\mathbb{P}^3\dasharrow Q\times\mathbb{P}^1$
by \cite[Proposition~A.4]{ReYou00}, see also \cite[Theorem~1.1.1]{CheltsovShramov}.
\end{proof}

We already described $\mathfrak{A}_5$-orbits in $\P^3$ of small length in Section~\ref{section:Hashimoto}.
Now let us describe $\mathfrak{A}_5$-invariant curves in $\P^3$ of degree less than eight.
As we will see a bit later, they all lie in the surface given by $F_3=0$,
and one of them also lies in the surface $Q$.
Thus, we need to take a closer look at these surfaces.

The surface $Q$ does not contain the $\mathfrak{A}_5$-orbits $\Sigma_5$, $\Sigma_{10}$, $\Sigma_{10}^\prime$, and $\Sigma_{15}$,
and it does contain  the $\mathfrak{A}_5$-orbits $\Sigma_{12}$ and $\Sigma_{12}^\prime$, see \cite[Lemma~5.3.3(vi)]{CheltsovShramov}.
Moreover, since $Q$ contains only two $\mathfrak{A}_5$-orbits of length $12$,
and $W_4$ is an irreducible representation of $\mathfrak{A}_5$,
it follows from~\mbox{\cite[Lemma~6.4.3(i)]{CheltsovShramov}} that the $\mathfrak{A}_5$-action on $Q$ is twisted diagonal,
i.e. the quadric $Q$ can be identified with $\mathbb{P}(U_2)\times\mathbb{P}(U_2^\prime)$.

Denote by $\mathrm{S}_3$ the surface in $\P^3$ that is given by $F_3=0$,
and denote by $\mathcal{B}_6$ the curve in $\P^3$ that is given by $F_2=F_3=0$.
Then $\mathrm{S}_3$ is a smooth surface known as the \emph{Clebsch diagonal cubic surface},
and $\mathcal{B}_6$ is a smooth irreducible curve of genus four  known as the \emph{Bring's curve}.

\begin{lemma}
\label{lemma:Bring}
The curve $\mathcal{B}_6$ is the only $\mathfrak{A}_5$-invariant curve in $Q$ of degree less than eight.
It contains the $\mathfrak{A}_5$-orbits $\Sigma_{12}$ and $\Sigma_{12}^\prime$. Moreover, the set
$\Sigma_{12}\cup\Sigma_{12}^\prime$ is cut out on $\mathcal{B}_6$ by the equation $F_4=0$.
\end{lemma}

\begin{proof}
By \cite[Lemma~5.1.5]{CheltsovShramov}, the curve $\mathcal{B}_6$ contains $\Sigma_{12}$ and $\Sigma_{12}^\prime$,
and does not contain other $\mathfrak{A}_5$-orbits of length less than $30$.
On the other hand, $\mathcal{B}_6$ is not contained in the surface given by $F_4=0$,
because the curve in $\P^3$ that is given by $F_2=F_4=0$ is a smooth curve of genus nine.
Therefore, the equation $F_4=0$ cuts out a subset of $\mathcal{B}_6$ that consists of
$$
4\mathrm{deg}(\mathcal{B}_6)=24
$$
points (counted with multiplicities).
Hence, we see that $\Sigma_{12}\cup\Sigma_{12}^\prime$ is cut out on $\mathcal{B}_6$ by the equation $F_4=0$.

Let $\Gamma$ be a curve in $Q$ of degree $d<8$. Then $d\geqslant 4$ by \cite[Lemma~5.3.3(ix)]{CheltsovShramov}.
Let us show that $\Gamma=\mathcal{B}_6$.
If $\Gamma$ is contained in $\mathrm{S}_3$, then $\Gamma=\mathcal{B}_6$ by construction.
Thus, we may assume that this is not the case.
Therefore, $d\ne 6$ and $d\ne 7$, because
$$
\mathrm{S}_3\cdot\Gamma=3d
$$
and the lengths of $\mathfrak{A}_5$-orbits in $\mathbb{P}^3$ are $5$, $10$, $12$, $15$, $20$, $30$, and $60$.
Hence, either $d=4$ or $d=5$.

The curve $\Gamma$ is a divisor of bi-degree $(a,b)$ on $Q\cong\mathbb{P}^1\times\mathbb{P}^1$, where $a$ and $b$ are non-negative integers.
Without loss of generality, we may assume that $a\leqslant b$.
Since
$$
a+b=d\in\{4,5\},
$$
the pair $(a,b)$ must be one of the following: $(0,4)$, $(1,3)$, $(2,2)$, $(0,5)$, $(1,4)$, or $(2,3)$.
The cases $(0,4)$ and $(0,5)$ are impossible by \cite[Lemma~6.4.1]{CheltsovShramov}.
Moreover,  $Q$ contains no $\mathfrak{A}_5$-invariant effective divisors of bi-degree $(1,3)$ and $(1,4)$ by \cite[Lemma~6.4.11(o)]{CheltsovShramov},
because the $\mathfrak{A}_5$-action on $Q$ is twisted diagonal.
Thus, either $(a,b)=(2,2)$ or $(a,b)=(2,3)$.
By \cite[Lemma~6.4.3(ii)]{CheltsovShramov}, the quadric $Q$ contains a smooth  rational curve $C_{1,7}$ that is a divisor of bi-degree $(1,7)$.
One has
$$
\Gamma\cdot C_{1,7}=7a+b\in\{16,18\},
$$
which is impossible, because the lengths of $\mathfrak{A}_5$-orbits in $C_{1,7}\cong\mathbb{P}^1$ are $12$, $20$, $30$, and~$60$.
\end{proof}

Let us describe four more $\mathfrak{A}_5$-invariant sextic curves contained in the surface~$\mathrm{S}_3$.
Recall from \cite[Lemma 6.3.3]{CheltsovShramov} that $\mathrm{S}_3$ contains two $\mathfrak{A}_5$-invariant curves
$\mathcal{L}_6$ and $\mathcal{L}_6^\prime$ such that each of them is a disjoint union of six lines,
and there is a commutative diagram
$$
\xymatrix{
&&\mathrm{S}_3\ar@{->}[dl]_{\pi}\ar@{->}[dr]^{\pi^\prime}&&\\
&\mathbb{P}^{2}\ar@{-->}[rr]^{\varsigma}&&\mathbb{P}^{2}&}
$$
where $\pi$ (respectively, $\pi^\prime$) is an $\mathfrak{A}_5$-birational morphism that contracts the lines of $\mathcal{L}_6$
(respectively, the lines of $\mathcal{L}_6^\prime$) to the unique $\mathfrak{A}_5$-orbit of length six in~$\P^2$. One has
$$
\mathcal{L}_6+\mathcal{L}_6^\prime\sim -4K_{\mathrm{S}_3}
$$
by construction. By \cite[Lemma 6.3.12(iii)]{CheltsovShramov}, each curve $\mathcal{L}_6$ and $\mathcal{L}_6^\prime$
contains a unique $\mathfrak{A}_5$-orbit of length $12$.
Moreover, the intersection
$\mathcal{L}_6\cap\mathcal{L}_6^\prime$
is an $\mathfrak{A}_5$-orbit of length $30$ by~\mbox{\cite[Lemma 6.3.12(vii)]{CheltsovShramov}}.
Thus, without loss of generality, we may assume that $\Sigma_{12}\subset \mathcal{L}_6$ and $\Sigma_{12}^\prime\subset\mathcal{L}_6^\prime$.

Similarly, denote by $\mathcal{C}$ (respectively, by $\mathcal{C}^\prime$) the smooth rational curve in $\mathrm{S}_3$
that is a proper transform of the unique $\mathfrak{A}_5$-invariant conic in $\P^2$ via~$\pi$ (respectively, via $\pi^\prime$). Then
$$
\mathcal{C}+\mathcal{C}^\prime\sim -4K_{\mathrm{S}_3},
$$
and both curves $\mathcal{C}$ and $\mathcal{C}^\prime$ are smooth rational sextic curves.
By construction, one has
$$
\mathcal{L}_6\cap\mathcal{C}=\mathcal{L}_6^\prime\cap\mathcal{C}^\prime=\varnothing.
$$
By \cite[Lemma~6.3.17]{CheltsovShramov}, one has
$\mathcal{L}_6\cap\mathcal{C}^\prime=\Sigma_{12}$, $\mathcal{L}_6^\prime\cap\mathcal{C}=\Sigma_{12}^\prime$,
and $\mathcal{C}\cap\mathcal{C}^\prime$ is an $\mathfrak{A}_5$-orbit of length~$30$.

\begin{lemma}
\label{lemma:A5-space}
The $\mathfrak{A}_5$-orbits of general points of the curves $\mathcal{B}_6$, $\mathcal{C}$, $\mathcal{C}^\prime$, $\mathcal{L}_6$,
and $\mathcal{L}_6^\prime$ are of length $60$.
These curves are the only $\mathfrak{A}_5$-invariant curves in $\mathbb{P}^3$ of degree less than eight.
\end{lemma}

\begin{proof}
The stabilizers in $\mathfrak{A}_5$ of general points of the curves $\mathcal{B}_6$, $\mathcal{C}$ and $\mathcal{C}^\prime$ are trivial,
because these curves are irreducible and none of them is contained in a plane in $\mathbb{P}^3$.
Thus, the $\mathfrak{A}_5$-orbits of general points of the curves $\mathcal{B}_6$, $\mathcal{C}$, and $\mathcal{C}^\prime$ are of length $60$.
The stabilizer of each irreducible component of the curve $\mathcal{L}_6$ acts faithfully on it by \cite[Corollary~5.2.3(v)]{CheltsovShramov}.
This implies that the $\mathfrak{A}_5$-orbit of a general point of (an irreducible component of) the curve $\mathcal{L}_6$ is of length $60$.
Similarly, the $\mathfrak{A}_5$-orbit of a general point of the curve $\mathcal{L}_6^\prime$ is also of length $60$.

Let us show that $\mathcal{B}_6$, $\mathcal{C}$, $\mathcal{C}^\prime$, $\mathcal{L}_6$,
and $\mathcal{L}_6^\prime$ are the only $\mathfrak{A}_5$-invariant curves in $\mathbb{P}^3$ of degree less than eight.
Let $\Gamma$ be an $\mathfrak{A}_5$-invariant curve in $\mathbb{P}^3$ of degree $d$.
If $\Gamma\subset\mathrm{S}_3$, then the required assertion follows from \cite[Theorem~6.3.18]{CheltsovShramov}.
Thus, we may assume that $\Gamma\not\subset\mathrm{S}_3$.
One has $\Gamma\not\subset Q$ by Lemma~\ref{lemma:Bring}.

By \cite[Lemma~5.3.3(ix)]{CheltsovShramov}, one has $d\geqslant 3$.
Hence either $d=4$ or $d=5$, because
$$
\mathrm{S}_3\cdot\Gamma=3d
$$
and the lengths of $\mathfrak{A}_5$-orbits in $\mathbb{P}^3$ are $5$, $10$, $12$, $15$, $20$, $30$, and $60$.
Thus, one has
$$
Q\cdot\Gamma=2d\in\{8,10\}.
$$
This is a contradiction, because the lengths of $\mathfrak{A}_5$-orbits in the quadric $Q$ are at least~$12$.
\end{proof}

\begin{corollary}
\label{corollary:A5-space}
There exists a unique $\mathfrak{A}_5$-invariant quartic surface $S_{\mathcal{L}_6}$
(respectively, $S_{\mathcal{C}}$) in $\P^3$ that contains the curve $\mathcal{L}_6$ (respectively, $\mathcal{C}$).
The surface $S_{\mathcal{L}_6}$ contains the curve $\mathcal{L}_6^\prime$ and
$$
\mathcal{L}_6+\mathcal{L}_6^\prime=\mathrm{S}_3\vert_{S_{\mathcal{L}_6}}.
$$
The surface $S_{\mathcal{C}}$ contains the curve $\mathcal{C}^\prime$ and
$$
\mathcal{C}+\mathcal{C}^\prime=\mathrm{S}_3\vert_{S_{\mathcal{C}}}.
$$
Moreover, the surface $S_{\mathcal{L}_6}$ (respectively, $S_{\mathcal{C}}$)
is the unique $\mathfrak{A}_5$-invariant quartic surface in~$\P^3$ that contains the curve $\mathcal{L}_6^\prime$
(respectively, $\mathcal{C}^\prime$).
Furthermore, both surfaces $S_{\mathcal{C}}$ and $S_{\mathcal{L}_6}$ are smooth.
\end{corollary}

\begin{proof}
Existence of the surfaces $S_{\mathcal{C}}$ and $S_{\mathcal{L}_6}$ follows from Lemma~\ref{lemma:A5-space},
because none of the curves  $\mathcal{C}$ and $\mathcal{L}_6$ is contained in $Q$ by Lemma~\ref{lemma:Bring}.
Moreover, the curve $\mathcal{C}$ (respectively, $\mathcal{L}_6$) is contained in the smooth locus of
the surface $S_{\mathcal{C}}$ (respectively, $S_{\mathcal{L}_6}$) by Proposition~\ref{proposition:Hashimoto}.
Thus, $\mathcal{C}^2=-2$ on the surface $S_{\mathcal{C}}$, and $\mathcal{L}_6^2=-12$ on the surface  $S_{\mathcal{L}_6}$.
Now it follows from \cite[Lemma~6.7.3(i),(ii)]{CheltsovShramov} that both surfaces $S_{\mathcal{C}}$ and $S_{\mathcal{L}_6}$ are smooth.

One has
$$
\mathrm{S}_3\vert_{S_{\mathcal{C}}}=\mathcal{C}+\Omega,
$$
where $\Omega$ is an effective $\mathfrak{A}_5$-invariant divisor on $S_{\mathcal{C}}$ of degree six.
By Lemma~\ref{lemma:A5-space}, the divisor $\Omega$ is one of the curves $\mathcal{B}_6$,
$\mathcal{C}$, $\mathcal{C}^\prime$, $\mathcal{L}_6$, or $\mathcal{L}_6^\prime$.
By Lemma~\ref{lemma:Bring}, one has $\Omega\ne\mathcal{B}_6$.
If $\Omega=\mathcal{C}$, then
$$
-2=\Omega\cdot\mathcal{C}=\left(\mathrm{S}_3\vert_{S_{\mathcal{C}}}-\mathcal{C}\right)\cdot\mathcal{C}=18-\mathcal{C}^2=20,
$$
which is absurd. Similarly, if $\Omega=\mathcal{L}_{6}$, then
$$
0=\Omega\cdot\mathcal{C}=\left(\mathrm{S}_3\vert_{S_{\mathcal{C}}}-\mathcal{C}\right)\cdot\mathcal{C}=20,
$$
which is absurd. Finally, if $\Omega=\mathcal{L}_{6}^\prime$, then $\Omega\cdot\mathcal{C}=24$ by \cite[Lemma~6.3.17]{CheltsovShramov},
so that
$$
24=\Omega\cdot\mathcal{C}=\left(\mathrm{S}_3\vert_{S_{\mathcal{C}}}-\mathcal{C}\right)\cdot\mathcal{C}=20,
$$
which is absurd again. Therefore, one has $\Omega=\mathcal{C}^\prime$.
Similarly, we see that $\mathcal{L}_6+\mathcal{L}_6^\prime=\mathrm{S}_3\vert_{S_{\mathcal{L}_6}}$.

Recall that the curve in $\mathbb{P}^3$ that is given by $F_2=F_4=0$  is a smooth curve of genus nine.
This implies all uniqueness assertions of the corollary.
\end{proof}

To prove Theorem~\ref{theorem:G-rigid}, we also need the following technical result.

\begin{lemma}
\label{lemma:isolation}
Let $\Sigma$ be an $\mathfrak{A}_5$-orbit on $\mathbb{P}^3$. Put
$$
m=\left\{%
\aligned
&2\ \text{if}\ \Sigma=\Sigma_5,\\
&3\ \text{if}\ \Sigma=\Sigma_{10},\ \text{or}\ \Sigma=\Sigma_{10}^\prime,\ \text{or}\ \Sigma=\Sigma_{15},\\
&4\ \text{if}\ \Sigma=\Sigma_{12}\ \text{or}\ \Sigma=\Sigma_{12}^\prime,\\
&10\ \text{if}\ |\Sigma|=20,\\
&15\ \text{if}\ |\Sigma|=30,\\
&30\ \text{if}\ |\Sigma|=60.\\
\endaligned\right.%
$$
Let $\mathcal{M}$ be the linear system consisting of all surfaces in $\mathbb{P}^3$ of degree $m$ that pass through the set $\Sigma$.
Then $\mathcal{M}$ does not have base curves and fixed components.
\end{lemma}

\begin{proof}
If $\Sigma=\Sigma_5$, then $\mathcal{M}$ does not have fixed components by \cite[Lemma~5.3.3(vi)]{CheltsovShramov},
and $\mathcal{M}$ does not have base curves, because $\mathbb{P}^3$ does not have $\mathfrak{A}_5$-invariant curves of degree less than six
by Lemma~\ref{lemma:A5-space}.
Similarly, if $\Sigma=\Sigma_{10}$ (respectively, $\Sigma=\Sigma_{10}^\prime$, $\Sigma=\Sigma_{15}$),
then $\Sigma$ is a singular locus of the nodal quartic surface $S_{10}$ (respectively, $S_{10}^\prime$, $S_{15}$).
Since the singular locus of a quartic surface is cut out by cubics,
we see that the base locus of $\mathcal{M}$ is $\Sigma$ in the latter case.
If $\Sigma=\Sigma_{12}$ or $\Sigma=\Sigma_{12}^\prime$,
then $\mathcal{M}$ does not have base curves and fixed components by Lemma~\ref{lemma:Bring}.
Thus, we may assume that $|\Sigma|\geqslant 20$, so that $m=\frac{|\Sigma|}{2}$.

Suppose that $\Sigma$ is not contained in the surface $Q$.
Let $P=[a_0:a_1:a_2:a_3:a_4]$ be a point in $\Sigma$.  Then $F_2(a_0,a_1,a_2,a_3,a_4)\ne 0$.
Put
$$
\alpha=\frac{F_5^2(a_0,a_1,a_2,a_3,a_4)}{F_2^5(a_0,a_1,a_2,a_3,a_4)},\
\beta=\frac{F_4(a_0,a_1,a_2,a_3,a_4)}{F_2^2(a_0,a_1,a_2,a_3,a_4)},\
\gamma=\frac{F_3^2(a_0,a_1,a_2,a_3,a_4)}{F_2^3(a_0,a_1,a_2,a_3,a_4)}.
$$
Then the system of equations
$$
\left\{%
\aligned
&F_5^2=\alpha F_2^5,\\
&F_4=\beta F_2^2,\\
&F_3^2=\gamma F_2^3,\\
\endaligned\right.%
$$
has finitely many solutions in $\mathbb{P}^3$, because the system of equations $F_2=F_3=F_4=F_5$ has no solution in $\mathbb{P}^3$.
Using the forms $F_5^2-\alpha F_2^5$, $F_4-\beta F_2^2$, and $F_3^2-\gamma F_2^3$, we can produce
three surfaces in $\mathcal{M}$ that have only finitely many common points.
This shows that the base locus of $\mathcal{M}$ is zero-dimensional as requested.

To complete the proof, we may assume that $\Sigma\subset Q$.
Then $\mathcal{M}$ contains divisors of the form $Q+R$,
where $R$ is any surface in $\mathbb{P}^3$ of degree $m-2$.
Thus, the base locus of $\mathcal{M}$ is contained in $Q$.
Denote by $T$ the hyperplane section of $Q$.
Let $\mathcal{M}_Q$ be the linear system consisting of all curves in $|mT|$ that pass through $\Sigma$.
Then $\mathcal{M}_Q$ is not empty, because
$$
h^0\left(Q,\mathcal{O}_{Q}(mT)\right)=(m+1)^2>|\Sigma|.
$$
Moreover, every curve in $\mathcal{M}_Q$ is cut out on $Q$ by a surface in $\mathcal{M}$,
because we have a surjection
$$
H^0\left(\mathbb{P}^3,\mathcal{O}_{\mathbb{P}^3}(m)\right)\twoheadrightarrow H^0\left(Q,\mathcal{O}_{Q}(mT)\right).
$$
Thus, $\mathcal{M}_Q$ is a restriction of the linear system $\mathcal{M}$ to the surface $Q$, and $Q$ is not a fixed surface of $\mathcal{M}$.
In particular, the base loci of $\mathcal{M}$ and $\mathcal{M}_Q$ are the same.

Suppose that the base locus of $\mathcal{M}_Q$ contains a curve.
Since $\mathcal{M}_Q$ is  $\mathfrak{A}_5$-invariant, the base locus of $\mathcal{M}_Q$ contains an $\mathfrak{A}_5$-invariant curve.
Let us denote it by $Z$.
Then $Z$ is a divisor of bi-degree $(a,b)$ on $Q\cong\mathbb{P}^1\times\mathbb{P}^1$, where $a\leqslant m$ and $b\leqslant m$.
One has
$$
(m-a+1)(m-b+1)=h^0\left(Q,\mathcal{O}_{Q}(mT-Z)\right)\geqslant h^0\left(Q,\mathcal{O}_{Q}(mT)\right)-|\Sigma|=(m+1)^2-|\Sigma|,
$$
which implies that
$$
|\Sigma|\geqslant am+mb+(a+b)-ab.
$$
Without loss of generality, we may assume that $a\leqslant b$.
This gives
$$
2m=|\Sigma|\geqslant mb,
$$
so that $a\leqslant b\leqslant 2$.
Thus, the degree of the curve $Z$ is $a+b\leqslant 4$, which contradicts Lemma~\ref{lemma:A5-space}.
\end{proof}

Now we are ready to prove Theorem~\ref{theorem:G-rigid}.

\begin{proof}[{Proof of Theorem~\ref{theorem:G-rigid}}]
By Theorem~\ref{theorem:X5}, the threefold $X_5$ is not $\mathfrak{A}_5$-birationally superrigid.
Thus, we may assume that $X\ne X_5$. Suppose that $X$ is not $\mathfrak{A}_5$-birationally superrigid.
Then it follows from \cite[Corollary~3.3.3]{CheltsovShramov} that there exists
an $\mathfrak{A}_5$-invariant mobile linear system $\mathcal{D}$ on $X$ such that
the singularities of the log pair $\left(X,\frac{2}{n}\mathcal{D}\right)$ are not canonical,
where $n$ is a positive integer that is defined by $\mathcal{D}\sim nH$.

Let $Z$ be an irreducible subvariety of $X$ that is a center of canonical singularities of the log pair $\left(X,\frac{2}{n}\mathcal{D}\right)$, see \cite[Definition~2.4.1]{CheltsovShramov}.
Then either $Z$ is a point, or $Z$ is an irreducible curve.
Let $\{Z_1,\ldots,Z_r\}$ be the  $\mathfrak{A}_5$-orbit of $Z=Z_1$.
If $Z$ is a curve, then
\begin{equation}
\label{equation:Skoda}
\mathrm{mult}_{Z_i}\left(D\right)>\frac{n}{2}
\end{equation}
for each $Z_i$ and a general surface $D$ in $\mathcal{D}$, see, for example, \cite[Lemma~2.4.4]{CheltsovShramov}.
If $Z$ is a smooth point, then
\begin{equation}
\label{equation:Iskovskikh}
\mathrm{mult}_{Z_i}\left(D_1\cdot D_2\right)>n^2
\end{equation}
for each $Z_i$ and two general surfaces $D_1$ and $D_2$ in~$\mathcal{D}$, see, for example, \cite[Theorem~2.5.2]{CheltsovShramov}.
We will use \eqref{equation:Skoda} to show that $Z$ is not a curve,
then we will use \eqref{equation:Iskovskikh} to show that $Z$ is not a smooth point of $X$.
Finally we will use \cite[Theorem~1.7.20]{CheltsovUMN} to exclude the case when $Z$ is a singular point of $X$.

Suppose that $Z$ is a curve.
Put $\Sigma=Z_1+\ldots+Z_r$.
For general surfaces $D_1$ and $D_2$ in~$\mathcal{D}$, the inequality \eqref{equation:Skoda} gives
\begin{multline*}
2n^2=H\cdot D_1\cdot D_2\geqslant\sum_{i=1}^{r}\mathrm{mult}_{Z_i}\left(D_1\cdot D_2\right)H\cdot Z_i \geqslant\\
\geqslant\sum_{i=1}^{r}\mathrm{mult}_{Z_i}\left(D_1\right)\mathrm{mult}_{Z_i}\left(D_2\right)H\cdot Z_i>\frac{n^2}{4}H\cdot\Sigma,
\end{multline*}
which implies that $H\cdot\Sigma<8$.
Applying Lemma~\ref{lemma:A5-space}, we see that
$$
H\cdot\Sigma=6,
$$
and $\tau(\Sigma)$ is one of the following curves: $\mathcal{B}_6$, $\mathcal{C}$, $\mathcal{C}^\prime$, $\mathcal{L}_6$, or $\mathcal{L}_6^\prime$.
Since $H\cdot\Sigma=6$ and $\tau(\Sigma)$ is a curve of degree six,
the double cover $\tau$ induces an isomorphism $\Sigma\xrightarrow{\sim}\tau(\Sigma)$.

We claim that $\tau(\Sigma)\ne \mathcal{B}_6$.
Indeed, $\mathcal{B}_6$ is not contained in $S$, and
$$
\mathcal{B}_6\cap S=\Sigma_{12}\cup\Sigma_{12}^\prime
$$
by Lemma~\ref{lemma:Bring}.
This shows that $\mathcal{B}_6$ intersects the surface $S$ transversally in $24$ points.
Hence, the preimage of the curve $\mathcal{B}_6$ via $\tau$ is a smooth irreducible curve
that is a double cover of $\mathcal{B}_6$ branched over $\Sigma_{12}\cup\Sigma_{12}^{\prime}$.
In particular, the preimage of the curve $\mathcal{B}_6$ via $\tau$ is not isomorphic to $\mathcal{B}_6$,
so that $\tau(\Sigma)\ne \mathcal{B}_6$.

The threefold $X$ is a quartic hypersurface in the weighted projective space $\mathbb{P}(1,1,1,1,2)$
with weighted homogeneous coordinates $x_1, x_2, x_3, x_4$ and $w$ defined by equation
$$
w^2=F_4-\lambda F_2^2.
$$
Denote by $\mathcal{P}$ the pencil that is cut out on $X$ by $\alpha w=\beta F_2$, where $[\alpha:\beta]\in\mathbb{P}^1$.

Let $Y([\alpha:\beta])$ be a surface in $\mathcal{P}$ corresponding to the point $[\alpha:\beta]$.
If $[\alpha:\beta]\ne [0:1]$, then $Y([\alpha:\beta])$ is mapped isomorphically by $\tau$ to the surface in $\mathbb{P}^3$ that is given by
$$
F_4=\left(\lambda+\frac{\beta^2}{\alpha^2}\right)F_2^2.
$$
If $[\alpha:\beta]=[0:1]$, then $\tau$ induces a double cover $Y([\alpha:\beta])\to Q$ branched over a curve
that is cut out by $F_4=0$ on the quadric $Q$.
Recall that this curve is smooth.
In particular, $Y([\alpha:\beta])$ is either a smooth $K3$ surface or a nodal $K3$ surface by Proposition~\ref{proposition:Hashimoto}.

Let $P$ be a general point in $\Sigma$.
Then its $\mathfrak{A}_5$-orbit is of length $60$,
because the $\mathfrak{A}_5$-orbit of the point $\tau(P)$ is of length $60$ by Lemma~\ref{lemma:A5-space}.
Let $Y$ be a surface in $\mathcal{P}$ such that $P\in Y$.
Then $\Sigma\subset Y$, since otherwise we would have
$$
12=Y\cdot\Sigma\geqslant 60.
$$
Hence, we have $\tau(\Sigma)\subset\tau(Y)$, which implies that $\tau(Y)\ne Q$,
because none of the curves $\mathcal{C}$, $\mathcal{C}^\prime$, $\mathcal{L}_6$, and $\mathcal{L}_6^\prime$
is contained in $Q$ by Lemma~\ref{lemma:Bring}.
Thus, $\tau(Y)$ is a (possibly nodal) $\mathfrak{A}_5$-invariant quartic surface, and $\tau$ induces an isomorphism $Y\xrightarrow{\sim}\tau(Y)$.

By Corollary~\ref{corollary:A5-space}, the surface $\tau(Y)$ is the surface $S_{\mathcal{C}}$ (respectively, $S_{\mathcal{L}_6}$)
in the case when $\tau(\Sigma)$ is one of the curves $\mathcal{C}$ or $\mathcal{C}^\prime$ (respectively, $\mathcal{L}_6$ or $\mathcal{L}_6^\prime$).
In particular, the surface $\tau(Y)$ is smooth.

By Corollary~\ref{corollary:A5-space}, the surface $Y$ contains an $\mathfrak{A}_5$-invariant curve $\Sigma^\prime$
such that $\Sigma^\prime\ne\Sigma$, the curves $\Sigma^\prime$ and $\Sigma$ are isomorphic, and
\begin{equation}
\label{equation:3H}
\Sigma+\Sigma^\prime\sim 3H\vert_{Y}.
\end{equation}
Indeed, we have
$\mathcal{L}_6+\mathcal{L}_6^\prime=\mathrm{S}_3\vert_{S_{\mathcal{L}_6}}$
and
$\mathcal{C}+\mathcal{C}^\prime=\mathrm{S}_3\vert_{S_{\mathcal{C}}}$
by Corollary~\ref{corollary:A5-space}.
Thus, if $\tau(\Sigma)=\mathcal{C}$ (respectively, $\tau(\Sigma)=\mathcal{C}^\prime$, $\tau(\Sigma)=\mathcal{L}_6$, $\tau(\Sigma)=\mathcal{L}_6^\prime$),
then we can let $\Sigma^\prime$ to be the preimage on $Y$ of the curve $\mathcal{C}^\prime$ (respectively, $\mathcal{C}$, $\mathcal{L}_6^\prime$, $\mathcal{L}_6$).

By construction, we have $\Sigma^2=(\Sigma^\prime)^2<0$.
On the other hand, we have
$$
\mathcal{M}\vert_{Y}=l\Sigma+l^\prime\Sigma^\prime+\Omega\sim nH\vert_{Y},
$$
where $l$ and $l^\prime$ are non-negative integers, and $\Omega$ is an effective divisor on the surface $Y$
whose support does not contain irreducible components of the curves $\Sigma$ and $\Sigma^\prime$.
By~\eqref{equation:3H}, we have
$$
(3l-n)\Sigma+(3l^\prime-n)\Sigma^\prime+3\Omega\sim 0.
$$
Moreover, one has $3l-n>0$, because $l>\frac{n}{2}$ by \eqref{equation:Skoda}.
Hence, we obtain $3l^\prime-n<0$, so that
$$
0\leqslant (3l-n)\Sigma\cdot\Sigma^\prime+3\Omega\cdot\Sigma^\prime=(n-3l^\prime)\left(\Sigma^\prime\right)^2<0,
$$
which is absurd. The obtained contradiction shows that $Z$ is not a curve.

We see that $Z$ is a point.
Denote by $\Xi$ its $\mathfrak{A}_5$-orbit.
Define an integer $m$ as in Lemma~\ref{lemma:isolation}.
Let $\mathcal{M}$ be the linear system consisting of all surfaces in $|mH|$ passing through $\Xi$.
Then $\mathcal{M}$ does not have base curves and fixed components by Lemma~\ref{lemma:isolation}.
Thus, if $Z$ is a smooth point of $X$, then \eqref{equation:Iskovskikh} gives
$$
n^2|\Xi|\geqslant 2mn^2=M\cdot D_1\cdot D_2\geqslant\sum_{i=1}^{r}\mathrm{mult}_{Z_i}\left(D_1\cdot D_2\right)>n^2|\Xi|
$$
for a general surface $M$ in $\mathcal{M}$, and two general surfaces $D_1$ and $D_2$ in $\mathcal{D}$.
Therefore, we see that $Z$ is a singular point of $X$.
By Proposition~\ref{proposition:Hashimoto}, either $X=X_{10}$, or $X=X_{10}^\prime$, or $X=X_{15}$.
Put $r=|\Xi|$, so that either $r=10$ or $r=15$.

Let $f\colon W\to X$ be a blow up of $\Xi$, and let $E_1,\ldots,E_{r}$ be the $f$-exceptional surfaces.
Denote by $\widetilde{D}_1$ and $\widetilde{D}_2$ the proper transforms via $f$ of the surfaces $D_1$ and $D_2$, respectively.
Then
$\widetilde{D}_1\sim \widetilde{D}_2\sim f^*(nH)-s\sum_{i=1}^{r}E_i$
for some positive integer $s$.
Denote by $\widetilde{M}$  the proper transform of the surface $M$ via $f$.
Then
$\widetilde{M}\sim f^*(3H)-t\sum_{i=1}^{r}E_i$
for some positive integer~$t$. Actually, one can show that $t=1$, but we will not use this.
Since $\mathcal{M}$ does not have base curves and fixed components, the divisor $\widetilde{M}$ is nef.
Thus, we have
$$
0\leqslant\widetilde{M}\cdot\widetilde{D}_1\cdot\widetilde{D}_2=6n^2-2rs^2t\leqslant 6n^2-2rs^2,
$$
which implies that $s\leqslant\frac{n\sqrt{3}}{\sqrt{r}}$.
On the other hand, $s>\frac{n}{2}$ by \cite[Theorem~1.7.20]{CheltsovUMN}.
This gives $r=10$, so that either $\tau(\Xi)=\Sigma_{10}$ or $\tau(\Xi)=\Sigma_{10}^\prime$.

By Remark~\ref{remark:10-points}, there exists a plane $\Pi\subset\mathbb{P}^3$
that contains at least four points of $\tau(\Xi)$ such that no three of them are collinear.
Let $C$ be a general conic in $\Pi$ that passes through these four points.
Then its preimage on $X$ via $\tau$ splits as a union of two smooth rational curves.
Denote by $\widetilde{C}$ the proper transform of one of these curves on $W$ via $f$.
Then $\widetilde{C}$ is not contained in the support of $\widetilde{D}_1$.
On the other hand, we have
$$
\widetilde{D}_1\cdot\widetilde{C}=\left(f^*(nH)-s\sum_{i=1}^{r}E_i\right)\cdot \widetilde{C}=2n-s\sum_{i=1}^{r}E_i\cdot\widetilde{C}\leqslant 2n-4s,
$$
which implies that $s\leqslant\frac{n}{2}$.
This is again impossible, because we know that~\mbox{$s>\frac{n}{2}$}.
\end{proof}

\end{document}